\documentclass[12pt]{article}
\usepackage{fullpage} 
\usepackage{parskip} 
\usepackage{tikz} 
\usepackage[light,math]{iwona}
\usepackage[T1]{fontenc}
\usepackage{amsmath}
\usepackage{amsfonts}
\usepackage{hyperref}
\usepackage{bbm}
\usepackage{mathabx}
\usepackage{url}
\usepackage{amsthm}
\usepackage{graphicx}
\usepackage{listings}
\usepackage{float}
\usepackage{tikz-cd}
\usepackage[pagewise]{lineno} 
\usepackage{imakeidx}
\makeindex
\usepackage{setspace} 
\usepackage{stmaryrd} 

\theoremstyle{plain}
\newtheorem{theorem}{Theorem}[section] 

\newtheorem{lemma}[theorem]{Lemma}
\newtheorem{proposition} [theorem]{Proposition}
\theoremstyle{definition} 
\newtheorem{defn}[theorem]{Definition}
\newtheorem{eg}[theorem]{Example}
\newtheorem{remark}[theorem]{Remark}

\newcommand{\PP}{\mathbb{P}}

\newcommand{\R}{\mathbb{R}}

\newcommand{\Z}{\mathbb{Z}}

\newcommand{\Ail}{\mathcal{A}}
\newcommand{\Bil}{\mathcal{B}}

\newcommand{\Fil}{\mathcal{F}}
\newcommand{\Gil}{\mathcal{G}}

\newcommand{\Nil}{\mathcal{N}}

\newcommand{\One}{\mathbbm{1}}

\title{The problem of infinite information flow}
\author{Zheng Bian\thanks{Clarkson
		Center for Complex Systems Science (C$^3$S$^2$), Potsdam, NY 13699 USA,
		(\texttt{zheng@bian-zheng.cn}).}
	\and Erik M. Bollt\thanks{Department of Electrical and Computer Engineering, Clarkson University, Potsdam, NY 13699 USA, and Clarkson
		Center for Complex Systems Science (C$^3$S$^2$), Potsdam, NY 13699 USA,
		(\texttt{ebollt@clarkson.edu}).}
}
\date{March 25, 2025}

\begin{document}
	\maketitle

\begin{abstract}
We study conditional mutual information (cMI) between a pair of variables $X,Y$ given a third one $Z$ and derived quantities including transfer entropy (TE) and causation entropy (CE) in the dynamically relevant context where $X=T(Y,Z)$ is determined by $Y,Z$ via a deterministic transformation $T$. Under mild continuity assumptions on their distributions, we prove a zero-infinity dichotomy for cMI for a wide class of $T$, which gives a yes-or-no answer to the question of information flow as quantified by TE or CE. Such an answer fails to distinguish between the relative amounts of information flow. To resolve this problem, we propose a discretization strategy and a conjectured formula to discern the \textit{relative ambiguities} of the system, which can serve as a reliable proxy for the relative amounts of information flow. We illustrate and validate this approach with numerical evidence.
\end{abstract}

%

	\tableofcontents

\section{Introduction}
Quantifying information flow is a critical task for understanding complex systems in various scientific disciplines, from neuroscience \cite{Vicente2011, Ursino2020, Shorten2021} to financial markets \cite{Dimpfl2013,Assaf2022}. Information measures such as mutual information (MI), conditional mutual information (cMI) \cite{Cover2005}, transfer entropy (TE) \cite{Schreiber2000}, and causation entropy (CE) \cite{Sun2015},   have become essential tools for this purpose.

Tracing back to the classic Weiner-Granger causality \cite{granger1969investigating,granger1988some,barnett2009granger,hendry2004nobel}, a central idea that underlies these information theoretic methods of quantifying information flow is the notion of \textit{disambiguation} in a predictive framework.
In contrast to the experimentalist approach, which infers causality from outcomes of perturbations and experiments, the predictive framework, which we consider below, is premised on alternative formulations of the forecasting question, with and without considering the influence of an external system. 

Formulated by Schreiber \cite{Schreiber2000} in 2000, TE is a quantitative attempt in this predictive framework. We think of $V=\{V_t\}$ and $U=\{U_t\}$ as stochastic processes indexed by discrete time $t=0,1,\cdots$; for a concrete example, imagine that $V,U$ record EEG times series data from different parts of the brain.  We expect that the present state $V_t$ informs about the future state $V_{t+1}$ and are interested in determining whether the present state $U_t$ also informs about $V_{t+1}$. 
If $V_{t+1}$ is conditionally independent of $U_t$ given $V_t$, then the knowledge about the state of $U_t$ does not resolve any uncertainty about the state of $V_{t+1}$, assuming one already has access to the state of $V_t$. In this case, we would like to conclude no information flow from $U$ to $V$ at time $t$ and zero TE accordingly. Otherwise, any deviation from this conditional independence indicates the presence of information flow, to be captured and quantified by some positive value of TE measured in bits per time unit.

By a slight generalization of Schreiber's original formulation and in agreement with the usual definition for discrete variables, we define TE 
\begin{equation}\label{eq:TE_defn}
	T_{U\to V,t} := I(V_{t+1}; U_t| V_t)
\end{equation}
to be the conditional mutual information  of $V_{t+1},U_{t}$ given $V_t$.
For simplicity, this is the case of lag length 1; longer lags are allowed in general. Causation entropy, proposed by Sun, Taylor and Bollt \cite{Sun2015}, generalizes TE to infer network connectivity \cite{sun2014causation,sun2014identifying,almomani2020entropic,lord2016inference}, by also building in conditioning on ternary influences as a way to resolve the differences between direct and indirect interactions. 
The precise definition of cMI will be given in Section \ref{sec:bg_cMI}. Roughly speaking, it quantifies the deviation from conditional independence of a pair of random variables conditioned on a third variable. 



\subsection{Zero-infinity dichotomy}
Consider a typical situation from dynamical systems, where the random variable $V_{t+1}$ is determined by 
$U_t,V_t$ via some deterministic map $T$, that is, 
\begin{equation} \label{eq:T_UV}
	V_{t+1}= T(U_t,V_t).
\end{equation}
If $V_{t+1}$ does not depend on $U_t$, that is, $V_{t+1}=T_0(V_t)$, then we trivially have zero information flow $T_{U\to V,t}=0$. In terms of probability distributions, this case corresponds to the regular conditional probability $\PP(V_{t+1}\in\cdot|V_t=v_t)=\delta_{T_0(v_t)}$ being a dirac delta.

Otherwise, one expects $T_{U\to V,t}>0$ to quantify the amount of information flowing from $U$ to $V$ at time $t$.
For example, if the map $T$ is highly ``ambiguous'', then the knowledge about the states of $U_t,V_t$ does not resolve much uncertainty about the state of $V_{t+1}$. 

\begin{eg}\label{eg:T_1_T_2}
	Consider two maps $T_1(u,v)=100(u+v)\mod1$ and $T_2(u,v)=u+v\mod1$. The knowledge about the states of $U_t,V_t$ up to $10^{-2}$ precision is completely lost via $T_1$ and trivially informs that $V_{t+1}=T_1(U_t,V_t)$ lies in $[0,1]$, whereas this knowledge under $T_2$ informs about the state of $V_{t+1}=T_2(U_t,V_t)$ up to precision $2\times 10^{-2}$.
	Therefore, we may expect $T_{U\to V,t}$ to be smaller in the more ambiguous case of $V_{t+1}=T_1(U_t,V_t)$ than in the case of $V_{t+1}=T_2(U_t,V_t)$. 
\end{eg}
However, under some mild continuity assumptions on the distribution of $V_{t+1}$, we see that in both cases, $T_{U\to V,t}=\infty$. This holds more generally for any measurable map $T$.
Throughout this paper, we assume that the random variables take values in standard measurable spaces, unless otherwise stated. This implies the existence and essential uniqueness of regular conditional probabilities and disintegrations; for details see Appendix \ref{sec:RCP_standard_space}.


\textbf{Theorem A (infinite information flow):} \textit{
	Assume that for a positive measure set of outcomes $v_t$ of $V_t$, the regular conditional probability distribution $\PP(V_{t+1}\in\cdot|V_t=v_t)$ of $V_{t+1}$ in Eq. (\ref{eq:T_UV}) 
	charges an atomless continuum. Then, the transfer entropy $T_{U\to V,t}$ from $U$ to $V$ at time $t$ is infinite.
}

\begin{remark}
	The positive measure set is with respect to the distribution of $V_t$.
	We say that a probability measure $\mu$ charges an atomless continuum if there is a measurable set $B$ such that $\mu(B)>0$ and $\mu(\{b\})=0$ for each point $b\in B$. The assumption of  Theorem A says that $V_t$ alone does not fully determine $V_{t+1}$ but rather leaves a rich continuum of possible values for $V_{t+1}$. This is the case, for example, when $V_{t+1}= T_1(U_t, V_t)$ or $V_{t+1}=T_2(U_t,V_t)$ as in Example \ref{eg:T_1_T_2} with $V_t$ and $U_t$ independent and following the uniform distribution on $[0,1]$.
\end{remark}


Theorem \ref{thm:cMI_0_infty} gives an equivalent but slightly different formulation of Theorem A and is proven in Section \ref{sec:infinite_cMI}.  The zero-infinity dichotomy of $T_{U\to V,t}$ gives a yes-or-no answer to the question of information flow. 

	A key step in the proof of Theorem A is to disintegrate the conditional mutual information into mutual information between conditioned variables. We believe that this result is interesting in its own right and state it below.

	\textbf{Theorem B (disintegration of conditional mutual information):} \textit{The conditional mutual information $I(X;Y|Z)$ of three random variables $X,Y,Z$ is the average of the mutual information $I(X_z;Y_z)$ between conditioned versions $X_z,Y_z$ of $X,Y$ defined in Eq. (\ref{eq:defn_X_z_Y_z}), that is,
		\[
		I(X;Y|Z)=\int I(X_z;Y_z) \mathrm{d} P_Z(z).
		\]}
	\begin{figure}[h]
		\centering
		\includegraphics[width=150mm]{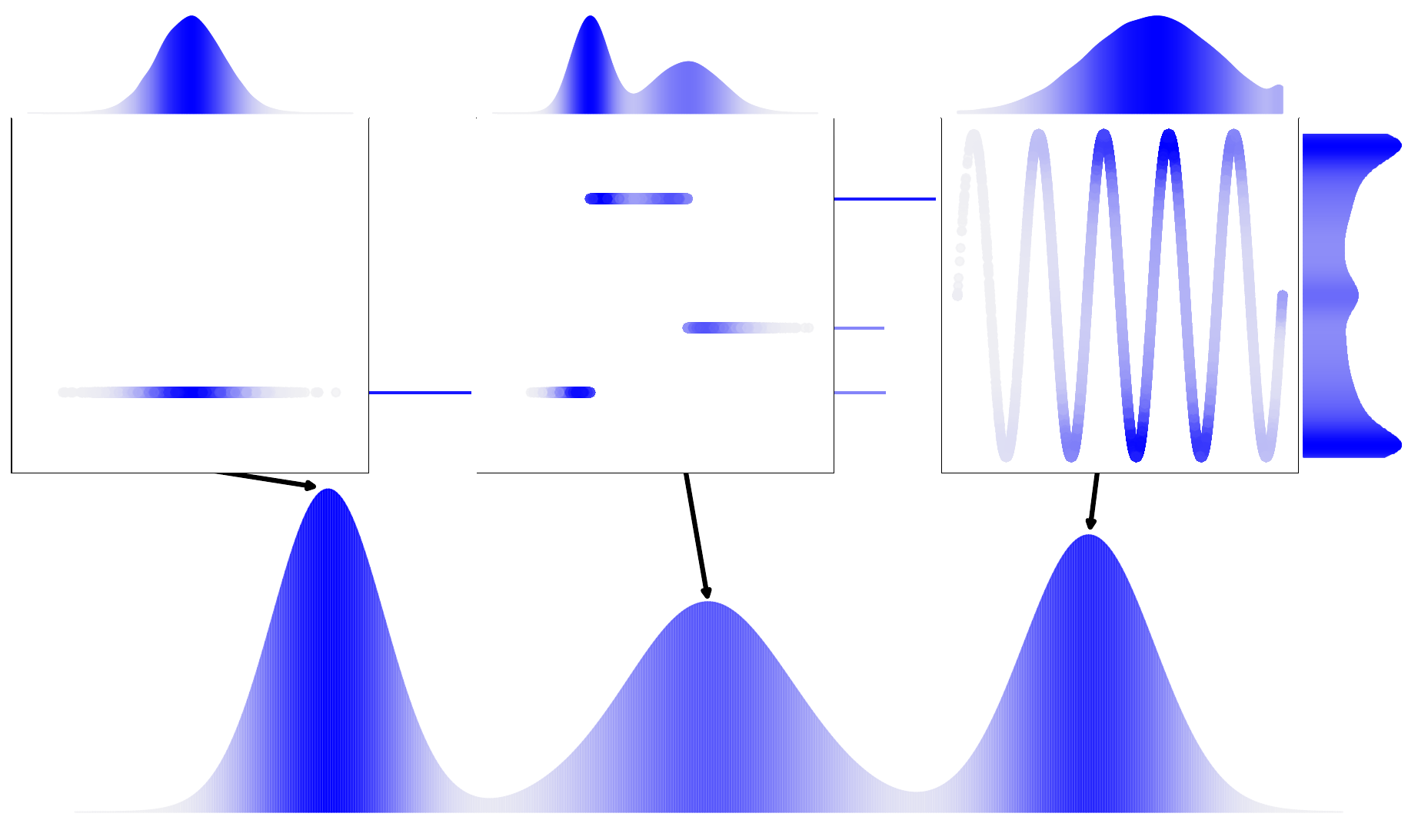}
		\caption{Disintegrated distributions. The main histogram at the bottom illustrates the distribution $P_Z$ of variable $Z$, which, together with $Y$, determines $X=T(Y,Z)$ via a measurable map $T$. The joint distribution $P_{XYZ}$ disintegrates into $(P_{XYZ})_z$ for each realization of $Z=z$, which can be interpreted as the joint distribution $P_{X_z Y_z}$ of the conditioned versions $X_z,Y_z$ of $X,Y$. 
			The left, center and right subplots above the main histogram illustrate three typical disintegrated distributions $(P_{XYZ})_z = P_{X_z Y_z}$, where $X_z$ follows a constant, atomic and continuous distribution, respectively. In each subplot, the scatter plot shows the joint distribution $P_{X_zY_z}$, the top histogram shows the marginal distribution $P_{Y_z}$, and the right histogram shows the marginal distribution $P_{X_z}$.  The intensity of the blue gradient indicates regions of high probability density.
		}
		\label{fig:disintegrated}
	\end{figure}
	\begin{remark}
The conditioned variables $X_z,Y_z$ describe the probabilistic landscape once the uncertainty about $Z$ is removed, by assuming that the outcome of $Z$ is $z$. This allows the intermediate measurement of $I(X_z;Y_z)$ on this particular outcome. By averaging across all outcomes of $Z$, the full conditional mutual information $I(X;Y|Z)$ is recovered. 
		We illustrate pictorially three typical scenarios 
        in Figure \ref{fig:disintegrated}; the subplots show the joint $P_{X_zY_z}$ and marginal distributions $P_{X_z}, P_{Y_z}$ of pairs of random variables $X_z,Y_z$ above the main histogram illustrating the distribution of $Z$.
		Proposition \ref{prop:average_disintegrated_MI} gives an equivalent but slightly different formulation of Theorem B and is proven in Section \ref{sec:conditional_MI_transfer_causation_entropy}. The main technical step involves the proper construction of $X_z,Y_z$ in Eq. (\ref{eq:defn_X_z_Y_z}) and the equivalence of disintegration and regular conditional probability in our context.
	\end{remark}

Theorem B reduces the analysis of TE or cMI in Theorem A to that of MI between conditioned variables. The exhaustive analysis of MI in the deterministic context thus completes the proof of Theorem A.

In practice, one computes TE from a finite amount of data and obtains finite positive values of $T_{U\to V,t}$.
As noted in \cite{Bossomaier2016}, much of the literature that applies TE to detect information flow focuses on establishing that $T_{U\to V,t}$ is statistically significantly different from zero, and treats the finite positive values of $T_{U\to V,t}$ as mere artifacts of finite sampling.

As discussed in Example \ref{eg:T_1_T_2}, a more ambiguous map such as $T_1$ allows through less information flow, which should be reflected by a smaller value of $T_{U\to V,t}$. Of course, this intuitive assumption is valid for discrete variables. However, it lacks theoretical justification in the case of continuous variables as pointed out by  Theorem A, which is typical for applications to dynamical systems. We refer to this discrepancy between the practically obtained finite TE values and the theoretic zero-infinity dichotomy as the \textit{problem of infinite information flow}.

\subsection{Resolution by discretization}
In light of Theorem B, it suffices to analyze the pairwise  $I(X;Y)$ for $X=T(Y)$, seeing that $I(X;Y|Z)$ can be obtained by averaging across $I(X_z,Y_z)$ for pairs of conditioned variables $X_z,Y_z$.
A resolution of the problem of infinite information flow needs to achieve two things:
\begin{enumerate}
	\item[(R1)] modify the model so as to obtain finite values for $I(X;Y)$,
	\item[(R2)] by comparing the relative values, distinguish between the relative amounts of information flow.
\end{enumerate}

By adding white noise to the map $T$ as employed in \cite{Surasinghe2020}, one can easily achieve (R1) as a blurring effect. However, we will show in Appendix \ref{sec:additive_noise} that this strategy still falls short of (R2). In fact,
	we prove for Bernoulli maps with uniformly distributed additive noise of amplitude $\epsilon$, uniformly distributed $Y$ and hence $X$, the resulting finite value of $I(X;Y)$ is $\ln\frac{1}{\epsilon}$, which is a function of the noise amplitude alone, independent of the expanding rate of the Bernoulli map. In this sense, the addition of white noise does not achieve (R2) because the resulting finite values of $I(X;Y)$ cannot distinguish between the relative dynamical ambiguities of the Bernoulli systems.

	We propose discretization as a strategy to achieve both (R1) and (R2) and illustrate in the one-dimensional case. 

	\textbf{Conjecture C (relative ambiguity of $(T,Y)$):} \textit{
		Suppose that $X,Y$ are $\R$-valued random variables with continuous probability density functions $f_X,f_Y$, respectively, and that there is a piecewise $C^1$ map $T$ for which $|T'|\geq 1$ and $X=T(Y)$. Consider the discretization by uniform mesh of size $\Delta>0$, that is,
		\[
		\Pi^{\Delta}:\R\to \Z\Delta,~~~~(\Pi^{\Delta})^{-1}\{i\Delta\} = [i\Delta,(i+1)\Delta),~~~~i\in\Z.
		\]
		Then, in the limit as $\Delta\to0^+$, the discretized variables $X^{\Delta}:=\Pi^{\Delta}X,Y^{\Delta}:=\Pi^{\Delta}Y$ satisfy
		\[
		I(X^{\Delta};Y^{\Delta}) + \ln\Delta \to  H(X) - \int \ln|T' | f_Y \mathrm{d}y =: -A_T(Y),
		\]
		where $H(X):= -\int f_X\ln f_X\mathrm{d}x$ is the differential entropy of $X$ and the quantity $A_T(Y)$ shall be called the \textit{relative ambiguity} of system $(T,Y)$.
	}

\begin{remark}
	In the special case of $T=\mathrm{id}$, we have $I(X^{\Delta};X^{\Delta})+\ln\Delta\to H(X)$ and recover the relation between Shannon entropy and differential entropy, see e.g. \cite[Section 9.3]{Cover2005}.
	More generally, it is clear that in the refinement limit of the discretization, i.e., as $\Delta\to 0^+$, the MI between the discretized variables $I(X^{\Delta},Y^{\Delta})$ tends to the infinite theoretic  value $I(X;Y)$. This is not our primary concern, however. What is more interesting is the behavior for finite $\Delta^{-1}$. Namely, for any finite $\Delta^{-1}$, the intuition that a more ambiguous system $(T,Y)$ with large relative ambiguity $A_T(Y)$ allows through less information is reflected by a smaller value of $I(X^{\Delta},Y^{\Delta})$.
	In this sense, discretization achieves both (R1) and (R2), resolving the problem of infinite information flow. 
	
Note that	the relative ambiguity $A_T(Y)$
	involves an entropy and an exponent, which naturally suggests a link to the Pesin entropy formula \cite{Pesin1977}. However, we defer further discussions on this link, as well as the proof and generalization of Conjecture C, to a separate ongoing work.
\end{remark}

Below, we validate Conjecture C with numerical evidence in some concrete dynamical examples. A sketch of the derivation of the conjectured formula for $A_T(Y)$ is included in the Appendix \ref{sec:app_discretized_MI_formula}.

\begin{eg}[Bernoulli interval maps] \label{eg:Bernoulli}
	Let the random variable $X=E_d(Y)$ be determined by $Y$ via the piecewise linear expanding map $E_d:[0,1]\to[0,1]$, $d\in\Z$, $d\geq 2$, on the unit interval given by 
	\[
	E_d(x)=d\cdot x\mod 1.
	\]
	Assume $Y$ follows a continuous distribution (we consider uniform and Gaussian $\Nil_{[0,1]}(0.3,0.02)$ centered at 0.3 with variance $0.02$ truncated between 0 and 1) on the interval.
	By Theorem A, or more directly, Theorem \ref{thm:MI_finite_infinite}, $I(X;Y)=\infty$.
	
		From Conjecture C, we have zero differential entropy of the uniformly distributed variable $X$ and a constant expansion rate $|T'|=d$, which yields	$A_T(Y)=\ln d$.
		
	A direct calculation, see Section \ref{sec:Bernoulli_maps}, shows that if $Y$ is uniformly distributed in $[0,1]$, then so is $X$ and 
		\begin{align*}
				I(X^{\Delta};Y^{\Delta}) = \ln \Delta^{-1} -\ln d = -\ln\Delta + A_T(Y),
			\end{align*}
in agreement with Conjecture C.


	In Figure \ref{fig:bernoulli}, we set $\Delta^{-1}=300$. The left and center panels show the scatter plots of the joint distribution $P_{X^{\Delta}Y^{\Delta}}$ of the discretized variables $X^{\Delta},Y^{\Delta}$, together with the marginal distribution $P_{Y^{\Delta}}$ on the top and $P_{X^{\Delta}}$ on the right of the scatter plots. We take $Y$ to follow the uniform distribution in the left panel in blue and the Gaussian $\Nil_{[0,1]}(0.3,0.02)$ in the center panel in red. The intensity of the colors indicates the high probability density. The right panel shows the mutual information $I(X^{\Delta},Y^{\Delta})$ decreases as the expansion rate $d$ of the Bernoulli map $E_d$ increases. The blue and red dots correspond to the cases of $Y$ following the $E_d$-invariant uniform distribution and the Gaussian $\Nil_{[0,1]}(0.3,0.02)$, respectively. For comparison, we superimpose the Conjecture prediction $\ln\Delta^{-1} +H(X) - \ln d$ in dashed lines.

	Observe that the dots from empirical calculations fit well with the Conjecture C predictions in dashed lines in both the uniform and Gaussian cases. In comparison to the uniform distribution, the tight Gaussian distribution $\Nil_{[0,1]}(0.3,0.02)$ of $Y$ results in a smaller (in fact, negative) differential entropy term $H(X)$ and hence a bigger relative ambiguity $A_T(Y)$ of the system $(T,Y)$ and a smaller discretized mutual information. As the Bernoulli expanding rate $d$ increases, the system $(T,Y)$ becomes more ambiguous in both the uniform and Gaussian cases, and hence $I(X^{\Delta},Y^{\Delta})$ decreases. For very large $d$, the expansion is so strong that even the tight Gaussian distribution of $Y$ smoothens to an almost uniform distribution of $X$ via $E_d$ and we see convergence of the two curves.
	 This example validates both Conjecture C and the discretization strategy's ability to achieve (R1--2).
	

	\begin{figure}[h]
		\centering
		\includegraphics[width=150mm]{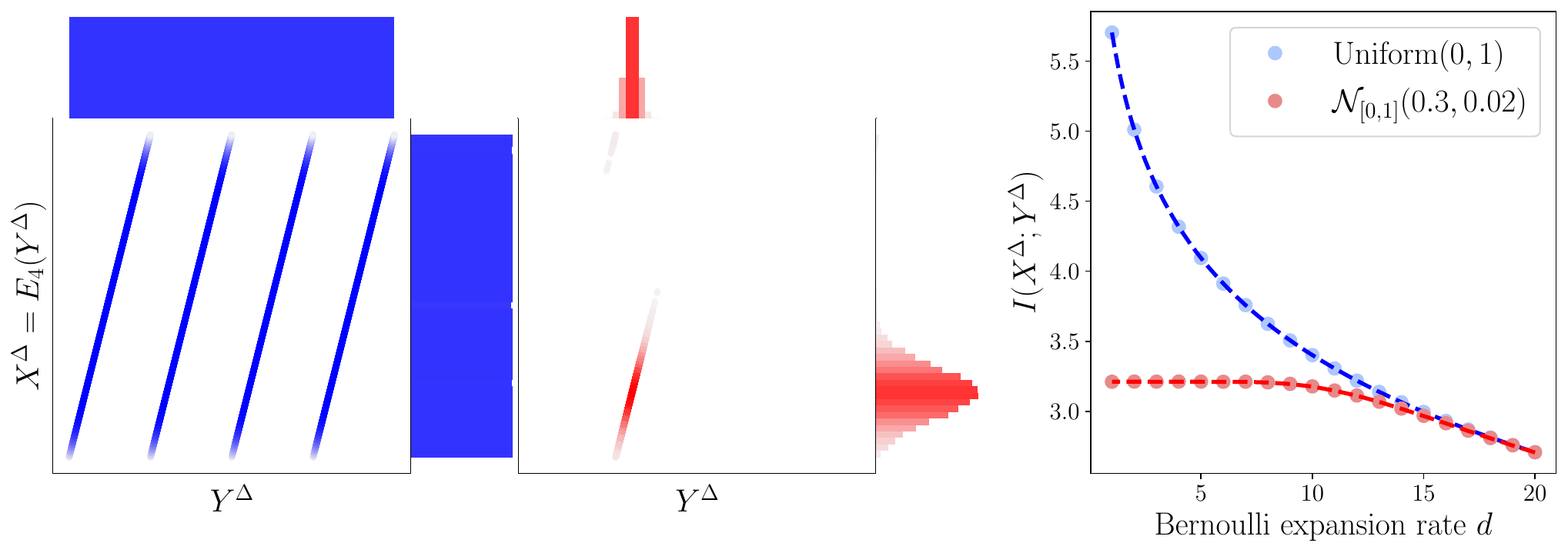}
		\caption{Discretization via uniform $\Delta^{-1}=300$ partition of continuous random variable $X=E_d(Y)$ determined by variable $Y$ via the Bernoulli map $E_d:x\mapsto d\cdot x\mod 1$. In the left and middle panels, the scatter plots show the joint distribution $P_{X^{\Delta}Y^{\Delta}}$ of the discretized variables $X^{\Delta},Y^{\Delta}$, together with the marginal distributions $P_{Y^{\Delta}}$ at the top and $P_{X^{\Delta}}$ on the right. The blue and red plots correspond to $Y$ following the uniform and Gaussian $\Nil_{[0,1]}(0.3,0.02)$ distributions, respectively. Here, $\Nil_{[0,1]}(0.3,0.02)$ means the Gaussian distribution centered at $0.3$ with variance $0.02$ and truncated between 0 and 1. The right panel plots for each Bernoulli expansion rate $d$, the corresponding $I(X^{\Delta};Y^{\Delta})$ of the discretized variables. The blue and red dots correspond to the empirical calculations of uniform and Gaussian $\Nil_{[0,1]}(0.3,0.02)$ distributions, respectively. The dashed lines show the theoretic predictions from Conjecture C.
		}
		\label{fig:bernoulli}
	\end{figure}
\end{eg}


The next example illustrates the discretization strategy in a nonlinear case and beyond the scope of Conjecture C (because the map has contracting regions).

\begin{eg}[Sine box functions] \label{eg:sine_box}
	Let the random variable $X=S_n(Y)$ be determined by $Y$ via the sine box function $S_n:[0,1]\to[0,1]$ given by
	\[
	S_n(x):=\frac{1+\sin2\pi nx}{2},~~~~n=1,2,\cdots.
	\]
	We consider two continuous distributions for $Y$, namely, the uniform distribution and the absolutely continuous $S_n$-invariant probability (acip) distribution. The acip is approximated by a long trajectory $\{y_t\}, y_{t+1}=S_n(y_t), t=\tau_0,\tau_0+1,\cdots,\tau_0+\tau-1$ of length $\tau=10^6$ with the first $\tau_0=1000$ iterates discarded as transients.  In both cases, we have $I(X;Y)=\infty$ by Theorem A, or more directly, Theorem \ref{thm:MI_finite_infinite}.
	
	In Figure \ref{fig:sine_box}, we discretize $X,Y$ the same way as in Example \ref{eg:Bernoulli}. For $S_4$, we show the scatter plots of $P_{X^{\Delta}Y^{\Delta}}$ and histogram of $P_{Y^{\Delta}}$ at the top and $P_{X^{\Delta}}$ on the right of the left and center panels. The uniform $P_Y$ shown in blue on the left is not invariant for $S_n$, but the red acip in the middle is $S_n$-invariant. The right panel shows that with $Y$ following either uniform or acip distribution, the mutual information $I(X^{\Delta};Y^{\Delta})$ between the discretized variables $X^{\Delta},Y^{\Delta}$ decreases as the function $S_n$ becomes more ambiguous (as $n$ increases). The calculation and simulation details are presented in Section \ref{sec:sine_box}.
	
	\begin{figure}[h]
		\centering
		\includegraphics[width=150mm]{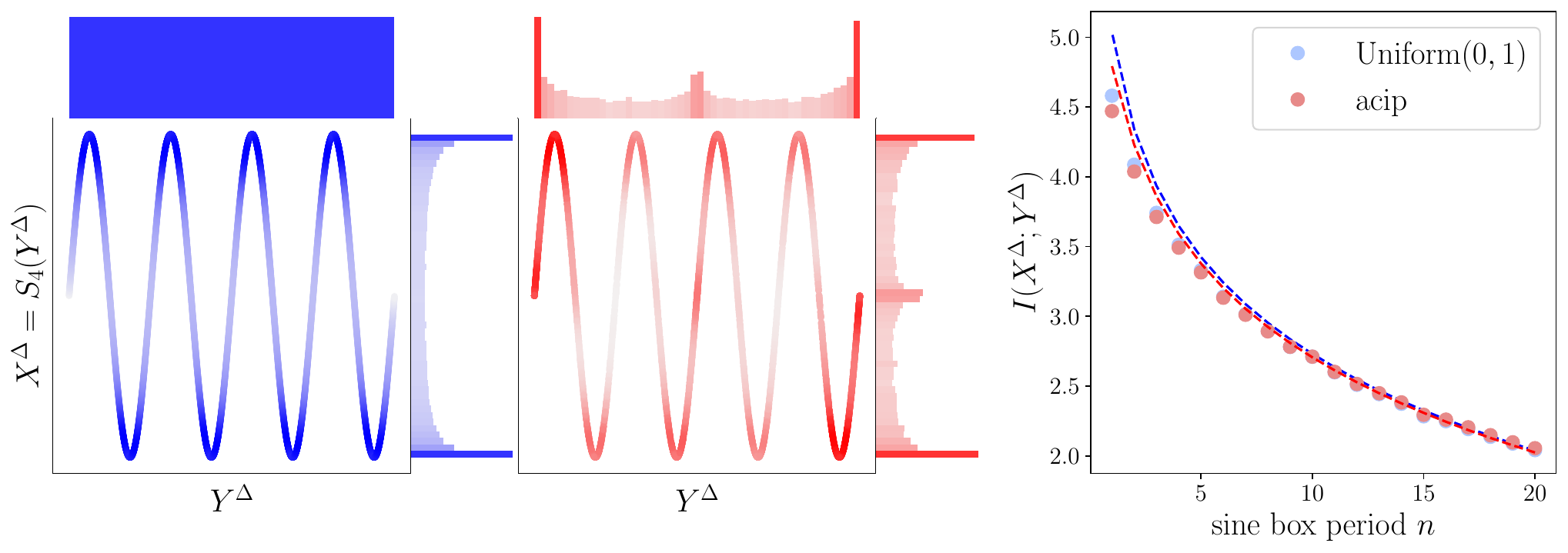}
		\caption{Discretization via uniform $\Delta^{-1}=300$ partition of continuous random variable $X=S_n(Y)$ determined by variable $Y$ via the sine box function $S_n:x\mapsto \frac{1+\sin 2\pi n x}{2}$. In the left and middle panels, the scatter plots show the joint distribution $P_{X^{\Delta}Y^{\Delta}}$ of the discretized variables $X^{\Delta},Y^{\Delta}$, together with the marginal distributions $P_{Y^{\Delta}}$ at the top and $P_{X^{\Delta}}$ on the right.  The right panel plots for each $n$, the corresponding MI $I(X^{\Delta};Y^{\Delta})$ of the discretized variables, with the empirical values shown in dots and Conjectured values in dashed lines. The blue and red colors correspond to $Y$ following the uniform and acip distributions, respectively.
		}
		\label{fig:sine_box}
	\end{figure}
\end{eg}

We remark that the sine box example falls outside the scope of Conjecture C because $S_n$ has contracting regions near $\frac{k}{2n} + \frac{1}{4n}$ for each $n=1,2,\cdots$ and $k=0,\cdots,2n-1$, where our Conjectured formula fails. It turns out that these contracting regions are assigned a higher weight  for smaller values of $n$ and uniform and acip densities of $f_Y$, leading to a bigger discrepancy between the empirical and Conjectured $I(X^{\Delta},Y^{\Delta})$ values for small $n$. In spite of this, it is remarkable that our formula still captures the trend that as $n$ increases, the relative ambiguity of $S_n$ increases and $I(X^{\Delta},Y^{\Delta})$ decreases. This example illustrates the validity of the discretization strategy.

To obtain meaningful finite values of TE or cMI in Eq. (\ref{eq:TE_defn}) that can distinguish the relative amounts of information flow, we discretize each conditioned version $Y_z$ and $X_z= T_z(Y_z)=T(Y_z,z)$ to obtain meaningful finite values of MI $I(X_z;Y_z)$ as in Examples \ref{eg:Bernoulli}, \ref{eg:sine_box}, and then average/integrate across all versions of $z$ against the marginal distribution $P_Z$ (or its discretization) in the sense of Theorem B. The numerical computations of TE in practice, in our view, essentially implement a similar discretization scheme.

\textbf{Organization of the paper.} 
In Section \ref{sec:bg_cMI} we review the definition and properties of MI and cMI and end with Proposition \ref{prop:average_disintegrated_MI} to decompose cMI into disintegrated MI of conditioned versions of the original variables. In Section \ref{sec:dyn_determinism}, we analyze the dichotomy properties of MI and cMI leading to the proof of the Theorem of infinite information flow. In Section \ref{sec:examples}, we present  detailed calculations and simulations for the illustrative Bernoulli and sine box examples \ref{eg:Bernoulli}, \ref{eg:sine_box}. In the Appendix, we discuss the key technical results on standard spaces, regular conditional probability, disintegration, and the effect of additive white noise.

\textbf{Acknowledgments.} We thank Tiago Pereira and Edmilson Roque dos Santos for helpful discussions and
comments.   Z.B. and E.M.B. are supported by the NSF-NIH-CRCNS.  E.M.B. is also supported by DARPA RSDN,
the ARO, and the ONR. 

\section{Background on cMI} \label{sec:bg_cMI}
We review notions and properties of 
Kullback-Leibler divergence in Section \ref{sec:KL}, 
entropy and mutual information in Section \ref{sec:entropy_MI}, 
and the conditional mutual information
in Section \ref{sec:conditional_MI_transfer_causation_entropy}. Some technical definitions and constructions, including the standard measurable space and regular conditional probability, are essential for the general definition of the conditional mutual information and therefore are also briefly reviewed in the Appendix.
More details can be found in \cite{Gray2009,Gray2011}.


Let $(\Omega,\Fil,\PP)$ be a probability space and $f:\Omega\to A$ a measurable function (also called random variable)
taking values in the measurable space $(A,\Bil_A)$ called the \textit{alphabet}. Denote
the \textit{distribution} of $f$  on $(A,\Bil_A)$ by
\[
P_f:= f_*\PP.
\]
When $A$ is a finite/countable set, we say that the alphabet is finite/discrete. 
For several random variables $f_1,\cdots,f_n$, we denote their joint distribution by $P_{f_1\cdots f_n}=(f_1,\cdots,f_n)_*\PP$ and the product measure of their marginal distributions by $P_{f_1}\otimes\cdots\otimes P_{f_n}=((f_1)_*\PP)\otimes\cdots\otimes ((P_{f_n})_*\PP)$.

\subsection{Kullback-Leibler divergence} \label{sec:KL}
First consider the special case where $A$ is a finite set and $\Bil_A=2^A$.
Given two probability measures $P,M$ on $(A,\Bil_A)$,
the \textit{Kullback-Leibler divergence of $P$ with respect to $M$} is defined to be
\[
\mathrm{KL}(P\|M):= \sum_{a\in A} P(a) \ln\frac{P(a)}{M(a)}.
\] 
Note that this makes sense only when $M(a)=0$ implies $P(a)=0$, i.e., $P\ll M$. In this case we define $0\ln \frac{0}{0}:=0$; otherwise, $\mathrm{KL}(P\|M)$ is defined to be $\infty$.


Now consider the general case: two probability measures $P,M$ on an arbitrary measurable space $(\Omega,\Fil)$. The \textit{Kullback-Leibler divergence} $\mathrm{KL}(P\| M)$ of $P$ with respect to $M$  is defined as
\[
\mathrm{KL}(P\| M) :=  \sup_f \mathrm{KL}(P_f\| M_f),
\]
where the 
supremum is taken over all 
random variables $f:\Omega\to A$ with a finite alphabet $A$. In fact, there is a sequence of random variables $f_n$ with finite alphabets, for example, obtained via increasingly fine partitions of $\Omega$, such that $\mathrm{KL}(P_{f_n}\| M_{f_n})$ tends to $\mathrm{KL}(P\|M)$ as $n\to\infty$; see \cite[Corollary 5.2.3]{Gray2011}.

\begin{remark}\label{rem:KL}
	KL is an asymmetric quantity that underlies the definitions of Shannon, transfer, causation entropy and (conditional) mutual information. 
\end{remark}

A key property is the so-called divergence inequality:

\begin{lemma}[Divergence inequality, \cite{Gray2011} Lemma 5.2.1] \label{lemma:KL_div_ineq}
	For any probability measures $P,M$ on a common alphabet, we have $\mathrm{KL}(P\|M)\geq 0$ and the equality holds precisely when $P=M$.
\end{lemma} 

Two cases of KL will be relevant to us.

\begin{lemma}[Relative entropy density \cite{Gray2011} Lemma 5.2.3] \label{lemma:relative_entropy_density} For any probability measures $P,M$ on a common alphabet, if $P\ll M$, then the Radon-Nikodym derivative $f:= \mathrm{d}P/\mathrm{d}M$ exists, is called the \textit{relative entropy density} of $P$ with respect to $M$,  and verifies
	\[
	\mathrm{KL}(P\|M)= \int_{\Omega} \ln f(\omega)\mathrm{d}P(\omega) = \int_{\Omega} f(\omega)\ln f(\omega)\mathrm{d}M(\omega).
	\]
	In this case, if $\Omega$ is finite then $f(\omega)= P(\omega)/M(\omega)$ and KL reduces to the finite alphabet case; if $\Omega=\R^d$ and $P,M\ll \mathrm{Leb}$ with densities $f_P,f_M$, respectively, then 
	\[
	\mathrm{KL}(P\|M) = \int_{\R^d} f(x) \ln \frac{f(x)}{g(x)} \mathrm{d}x.
	\]
	On the other hand, if $P$ is not absolutely continuous with respect to $M$, then
	\[
	\mathrm{KL}(P\|M)=\infty.
	\]
\end{lemma}


	%
	%
%

\subsection{Mutual information} \label{sec:entropy_MI}
Define
the \textit{mutual information} between two random variables $X$ and $Y$ to be
\[
I(X;Y) := \mathrm{KL}(P_{XY}\| P_X\otimes P_Y).
\]
It can be shown \cite[Chapter 2.5]{Gray2011} that the (Shannon) \textit{entropy} of $X$ (defined as $H(X):= -\sum_{x\in A_X} p_X(x) \ln p_X(x)$ in the discrete alphabet case) can be recovered by the mutual information with $X$ itself $H(X)=I(X;X)$ and therefore $I(X;Y) = H(X) + H(Y) - H(X,Y)$.



%


%

\begin{remark}\label{rem:mutual_information_deviation_indep}
	In light of Lemma \ref{lemma:KL_div_ineq}, it is clear that $I(X;Y)$ equals zero precisely when $X,Y$ are independent and quantifies their deviation from independence otherwise. The product of marginals $P_X\otimes P_Y$ serves as the reference independent model against which the joint distribution $P_{XY}$ is compared. More precisely, if $(X',Y')$ has joint distribution $P_X\otimes P_Y$, then $X',Y'$ are independent and have same marginal dsitributions as $X,Y$.
\end{remark}

\subsection{Conditional mutual information} \label{sec:conditional_MI_transfer_causation_entropy}
First we consider the finite alphabet case: three random variables $X,Y,Z$ with finite alphabets $A_X,A_Y,A_Z$, each equipped with the power-set $\sigma$-algebra $\Bil_{A_*}=2^{A_*}$, $*=X,Y,Z$. 
Define the \textit{conditional mutual information of $X,Y$ given $Z$} to be
\begin{equation}\label{eq:cMI_defn}
	I(X;Y|Z) := \mathrm{KL}(P_{XYZ}| P_{X\times Y|Z}),
\end{equation}
where $P_{X\times Y|Z}$ is a probability distribution on $A_X\times A_Y\times A_Z$ defined by
\begin{equation}\label{eq:cMI_discrete_defn}
	P_{X\times Y|Z} (B_X \times B_Y\times B_Z) := \sum_{z\in B_Z} \PP(X\in B_X| Z=z) \PP(Y\in B_Y| Z=z) \PP(Z=z)
\end{equation}
for any $B_X\in\Bil_{A_X}$, $B_Y\in\Bil_{A_Y}$ and $B_Z\in\Bil_{A_Z}$. Here, the conditional probability is the usual one $\PP(F|E)=\frac{\PP(F\cap E)}{\PP(E)}$ provided that $\PP(E)>0$.
\begin{remark}\label{rem:cMI}
	As discussed in the Introduction, conditional mutual information is designed to quantify the deviation from conditional independence of $X,Y$ given $Z$. And $P_{X\times Y|Z}$ is designed to serve as the conditional independent model against which to compare the joint distribution $P_{XYZ}$, cf. the role of $P_{X}\otimes P_Y$ in the definition of $I(X;Y)$ as discussed in Remark \ref{rem:mutual_information_deviation_indep}.
	More precisely, consider new random variables $X',Y',Z'$ with joint distribution $P_{X\times Y|Z}$ and observe
	\begin{itemize}
		\item $X',Y',Z'$ have the same marginal distributions as $X,Y,Z$: $P_{X'}=P_X$, $P_{Y'}=P_Y$, $P_{Z'}=P_Z$;
		\item $X',Y'$ have the same conditional marginal distributions given $Z'$ as $X,Y$ given $Z$: $\PP(X\in B_X|Z=z)=\PP(X'\in B_X|Z'=z)$ and 
		$\PP(Y\in B_Y|Z=z)= \PP(Y'\in B_Y|Z'=z)$;
		\item $X',Y'$ are conditionally independent given $Z'$: $\PP(X'\in B_X,Y'\in B_Y|Z'=z)= \PP(X'\in B_X|Z'=z)\PP(Y'\in B_Y|Z=z)$.
	\end{itemize}
	In other words, $P_{X\times Y|Z}$ is a ``Markovization'' of the joint distribution $P_{XYZ}$ in the sense that the modified random variables $X',Y',Z'$ form a Markov chain $Y'\to Z'\to X'$ (or $X'\to Z'\to Y'$) because the information about the state of $Y'$, in addition to that of $Z'$, does not further resolve the uncertainty about the state of $X'$ (the same holds with $X',Y'$ swapped).
\end{remark}



To generalize the definition of $I(X;Y|Z)$ in Eq. \ref{eq:cMI_defn}, the main challenge lies with the conditional probabilities appearing in the definition (\ref{eq:cMI_discrete_defn}) of $P_{X\times Y|Z}$. 
In general, we may well have $\PP(Z=z)=0$ for each $z\in A_Z$, for example, take $Z$ to be uniformly distributed on $A_Z=[0,1]$, or any other distribution absolutely continuous with respect to Lebesgue. This makes it impossible to define $\PP(F|Z=z)$ in the same way as the discrete alphabet case $\frac{\PP(F\cap \{Z=z\})}{\PP(Z=z)}$. 

This challenge can be met by (i) interpreting the conditional probability $\PP(X\in B_X|Z=z)$, rather than a fraction, as a Radon-Nikodym derivative for fixed $B_X\in\Bil_{A_X}$  and (ii) requiring that the alphabets of $X,Y,Z$ be ``standard'' measurable spaces so that $\PP(X\in B_X|Z=z)$ is well-defined as regular conditional probability simultaneously for all $B_X\in\Bil_{A_X}$ and similarly for $\PP(Y\in B_Y|Z=z)$. In \cite{Gray2011}, there is an even more general definition beyond standard alphabets. Since the standard alphabet already covers the practically relevant cases such as Polish spaces, we shall contain our discussion in the standard alphabet case and leave the details in the Appendix.

Consider three random variables $X,Y,Z$ on a common probability space $(\Omega,\Fil,\PP)$ with standard alphabets $(A_X,\Bil_{A_X})$, $(A_Y,\Bil_{A_Y})$, $(A_Z,\Bil_{A_Z})$, respectively. 
See Appendix \ref{sec:RCP_standard_space} for details. Define the \textit{conditional average mutual information} as in Eq. (\ref{eq:cMI_defn}) where the Markovization $P_{X\times Y|Z}$ is given in terms of regular conditional probabilities, for $B_X\in\Bil_{A_X}$, $B_Y\in\Bil_{A_Y}$, $B_Z\in \Bil_{A_Z}$
\begin{align*}
	P_{X\times Y|Z}(B_X\times B_Y\times B_Z):=& \int_{Z^{-1}B_Z} \PP(X\in B_X|\sigma(Z)) \PP(Y\in B_Y| \sigma(Z)) \mathrm{d}\PP \\
	=& \int_{B_Z} \PP(X\in B_X|Z=z) \PP(Y\in B_Y|Z=z) \mathrm{d}P_Z(z).
\end{align*}

\begin{remark}
	Note that $P_{X\times Y|Z}$ is a deterministic probability measure and hence the conditional mutual information $I(X;Y|Z)$ is a deterministic object on $A_X\times A_Y\times A_Z$, even though the notation suggests some conditioning. As the construction above shows, the randomness from conditioning on $Z$ is averaged out.
	
	In light of Lemma \ref{lemma:KL_div_ineq}, $I(X;Y|Z)$ equals zero precisely when $X,Y$ are conditionally independent given $Z$ and quantifies the deviation from this conditional independence otherwise.
\end{remark}

Since $P_{XYZ},P_{X\times Y|Z}$ both have $Z$-marginals equal to $P_Z$ by construction, they admit disintegrations with respect to $P_Z$ denoted by $(P_{XYZ})_z$ and $(P_{X\times Y|Z})_z$, which coincide with the regular conditional probabilities: for $P_Z$-a.e. $z\in A_Z$, and all $B_X\in\Bil_{A_X},B_Y\in\Bil_{A_Y}$, we have
\begin{align*}
	(P_{XYZ})_z(B_X\times B_Y) =& \PP(X\in B_X,Y\in B_Y|Z=z),\\
	(P_{X\times Y|Z})_z(B_X\times B_Y) =& \PP(X\in B_X|Z=z)\PP(Y\in B_Y|Z=z).
\end{align*}
See Appendix \ref{sec:RCP_standard_space} for more details.
We will sometimes prefer the disintegration notation to the regular conditional probability notation for clarity of presentation. 

\begin{defn}[$Z$-conditioned random variables]
For each $z\in A_Z$, define the \textit{$Z$-conditioned random variables $X_z,Y_z$} with alphabets $(A_X,\Bil_{A_X}),(A_Y,\Bil_{A_Y})$, respectively, and joint distribution 
\begin{equation}\label{eq:defn_X_z_Y_z}
	P_{X_zY_z}:=(P_{XYZ})_z,~~~~z\in A_Z.
\end{equation}
\end{defn}

Then, their marginal distributions are given by
\begin{align*}
	P_{X_z}(B_X)=& P_{X_zY_z}(B_X\times A_Y) = \PP(X\in B_X|Z=z),~~~~z\in A_Z,B_X\in \Bil_{A_X}\\
	P_{Y_z}(B_Y) =& P_{X_zY_z}(A_X\times B_Y) = \PP(Y\in B_Y|Z=z),~~~~z\in A_Z,B_Y\in \Bil_{A_Y}.
\end{align*}
Hence,
\[
(P_{X\times Y|Z})_z = P_{X_z}\otimes P_{Y_z}
\]
and
\[
I(X_z;Y_z) = \mathrm{KL}(P_{X_zY_z}\|P_{X_z}\otimes P_{Y_z}) = \mathrm{KL}((P_{XYZ})_z\| (P_{X\times Y|Z})_z).
\]
The intuition behind the above construction of $X_z,Y_z$ is to consider them as the disintegrated versions of $X,Y$ on the $z$-slice $A_X\times A_Y\times\{z\}$. The next proposition shows that the conditional mutual information $I(X;Y|Z)$ is the average of $I(X_z;Y_z)$ across all such $z$-slices.

\begin{proposition}[Average of disintegrated MI] \label{prop:average_disintegrated_MI}
	Consider three random variables $X,Y,Z$ on a common probability space $(\Omega,\Fil,\PP)$ with standard alphabets $(A_X,\Bil_{A_X})$, $(A_Y,\Bil_{A_Y})$, and $(A_Z,\Bil_{A_Z})$, respectively. Then, the conditional mutual information $I(X;Y|Z)$ is the $P_Z$-average of mutual information $I(X_z;Y_z)=\mathrm{KL}((P_{XYZ})_z\| (P_{X\times Y|Z})_z)$ between the $z$-conditioned random variables $X_z,Y_z$.
	More precisely, if $(P_{XYZ})_z \ll (P_{X\times Y|Z})_z$ for $P_Z$-a.e. $z\in A_Z$, then the Radon-Nikodym derivative 
	\[
	\frac{\mathrm{d}(P_{XYZ})_z}{\mathrm{d}(P_{X\times Y|Z})_z}(x,y) = \frac{\mathrm{d}P_{XYZ}}{\mathrm{d}P_{X\times Y|Z}}(x,y,z)~~~~\text{for $P_{X\times Y|Z}$-a.e. $(x,y,z)$}
	\]
	and hence
	\[
	I(X;Y|Z) = 	\int_{A_Z} I(X_z;Y_z)
	\mathrm{d}P_Z(z);
	\]
	otherwise, there is $B_Z\in\Bil_{A_Z}$ with $P_Z(B_Z)>0$ and $I(X_z;Y_z)=\mathrm{KL}\left((P_{XYZ})_z \| (P_{X\times Y|Z})_z\right) =\infty$ for each $z\in B_Z$, in which case  $I(X;Y|Z)=\infty$.
\end{proposition}



\textit{Proof.}
First consider $(P_{XYZ})_z \ll (P_{X\times Y|Z})_z$ for $P_Z$-a.e. $z\in A_Z$. Then the Radon-Nikodym derivative $\frac{\mathrm{d} (P_{XYZ})_z}{\mathrm{d} (P_{X\times Y|Z})_z}$ exists for $P_Z$-a.e. $z\in A_Z$. Integrating its logarithm against $(P_{XYZ})_z$ yields, according to Lemma \ref{lemma:relative_entropy_density},
\begin{align*}
	\mathrm{KL}\left((P_{XYZ})_z \| (P_{X\times Y|Z})_z\right) = \int_{A_X\times A_Y} \ln	\frac{\mathrm{d} (P_{XYZ})_z}{\mathrm{d} (P_{X\times Y|Z})_z} \mathrm{d}(P_{XYZ})_z.
\end{align*}
Further integrating the above equation against $P_Z$ yields, by definition of disintegration,
\begin{align*}
	\int_{A_Z}\mathrm{KL}\left((P_{XYZ})_z \| (P_{X\times Y|Z})_z\right) \mathrm{d}P_Z(z) =&\int_{A_Z} \int_{A_X\times A_Y} \ln	\frac{\mathrm{d} (P_{XYZ})_z}{\mathrm{d} (P_{X\times Y|Z})_z} \mathrm{d}(P_{XYZ})_z \mathrm{d}P_Z(z)\\
	=& \int_{A_X\times A_Y\times A_Z} \ln	\frac{\mathrm{d} (P_{XYZ})_z}{\mathrm{d} (P_{X\times Y|Z})_z}  \mathrm{d}P_{XYZ}.
\end{align*}
For any $B_X\in \Bil_{A_X},B_Y\in\Bil_{A_Y},B_Z\in\Bil_{A_Z}$, by definition of disintegration, we have
\begin{align*}
	&\int_{B_X\times B_Y\times B_Z} \frac{\mathrm{d}(P_{XYZ})_z}{\mathrm{d}(P_{X\times Y|Z})_z}(x,y) \mathrm{d}P_{X\times Y|Z} (x,y,z) \\
	=& \int_{B_Z} \int_{B_X\times B_Y} \frac{\mathrm{d}(P_{XYZ})_z}{\mathrm{d}(P_{X\times Y|Z})_z}(x,y) \mathrm{d}(P_{X\times Y|Z})_z (x,y) \mathrm{d}P_Z(z)\\
	=& \int_{B_Z} (P_{XYZ})_z(B_X\times B_Y) \mathrm{d}P_Z(z)\\
	=& P_{XYZ}(B_X\times B_Y\times B_Z).
\end{align*}
By uniqueness of Radon-Nikodym derivative, we conclude
\[
\frac{\mathrm{d}(P_{XYZ})_z}{\mathrm{d}(P_{X\times Y|Z})_z}(x,y) = \frac{\mathrm{d}P_{XYZ}}{\mathrm{d}P_{X\times Y|Z}}(x,y,z)~~~~\text{for $P_{X\times Y|Z}$-a.e. $(x,y,z)$}.
\]
We continue
\begin{align*}
	\int_{A_Z}\mathrm{KL}\left((P_{XYZ})_z \| (P_{X\times Y|Z})_z\right) \mathrm{d}P_Z(z) 
	=& \int_{A_X\times A_Y\times A_Z} \ln	\frac{\mathrm{d} P_{XYZ}}{\mathrm{d} P_{X\times Y|Z}}  \mathrm{d}P_{XYZ} = I(X;Y|Z),
\end{align*}
by Lemma \ref{lemma:relative_entropy_density}. 
This proves the first assertion.

Now we consider the other case: there is some $B_Z$ with $P_Z(B_Z)>0$ and for each $z\in B_Z$, there is some $B_{XY}^{(z)}$ with $(P_{XYZ})_z(B_{XY}^{(z)})>0=(P_{X\times Y|Z})_z$, then the set
\[
B_{XYZ}:= \bigcup_{z\in B_Z} B_{XY}^{(z)}\times\{z\}
\]
has the property that
\[
P_{XYZ}(B_{XYZ}) >0=P_{X\times Y|Z}(B_{XYZ}).
\]
In particular, $P_{XYZ}$ is not absolutely continuous with respect to $P_{X\times Y|Z}$. Hence, by Lemma \ref{lemma:relative_entropy_density}, we have
\[
I(X;Y|Z) = \mathrm{KL}(P_{XYZ}\|P_{X\times Y|Z})=\infty.  \qed
\]

\begin{eg}[Transfer and causation entropy]
	Consider a stochastic process $X=(X_0,X_1,\cdots)$ taking values in $(A_X,\Bil_{A_X})$ and another stochastic process $Y=(Y_0,Y_1,\cdots)$ taking values in $(A_Y,\Bil_{A_Y})$. 
	
	As in \cite[Chapter 9.8.1]{Bollt2013}, we are interested to quantify the information flow from $Y$ to $X$ at time $t$, conditioned on some history $X_t^{(k)}=(X_t,\cdots,X_{t-k+1})$ of $X$ itself $k$-steps into the past. If there is no such information flow, then $X_{t+1},Y_t^{(l)}$ should be conditionally independent given $X_t^{(k)}$, i.e.,
	\[
	I(X_{t+1}; Y_t^{(l)}|X_t^{(k)})=0.
	\]
	Otherwise, the information flow can be quantified by the deviation from conditional independence. This motivates our definition
	\[
	T_{Y\to X,t}:= I(X_{t+1}; Y_t^{(l)}|X_{t}^{(k)}) = \mathrm{KL}\left( \left. P_{X_{t+1}X_{t}^{(k)}Y_t^{(l)}} \right\| P_{X_{t+1}\times Y_{t}^{(l)}|X_{t}^{(k)}} \right),
	\]
	which has a similar form to the discrete version given by eq. (9.128) in \cite{Bollt2013}. Variations such as unlimited memory can also be considered.

	The causation entropy is a fruitful generalization of TE in the context of a network of stochastic processes $X^v$ indexed by nodes $v\in V:=\{1,\cdots,n\}$, where each $X^v = (X^v_t)_{t=0,1,\cdots}$. Given three collections $I,J,K\subseteq V$ of nodes, CE (looking 1 step into the past)
	\cite{Sun2015} is defined to be
	\[
	C_{J\to I|K,t}: = I(X_{t+1}^{(I)};X_{t}^{(J)}|X_{t}^{(K)}) = \mathrm{KL} \left(\left. P_{X_{t+1}^{(I)} X_t^{(K)} X_t^{(J)}}\right\| P_{X_{t+1}^{(I)}\times X_t^{(J)}| X_t^{(K)}}\right).
	\]
	In a discovery algorithm, \cite{Sun2015} uses $C_{J\to I|K,t}$ to quantify the information flowing from nodes $J$ to nodes $I$ conditioned on nodes $K$, where $I$ nodes are the potential neighbors of $J$ nodes under consideration and $K$ is the collection of known neighbors of $J$; the authors identify the most likely neighbors of $J$ as the collection $I$ that maximizes $C_{J\to I|K,t}$.
\end{eg}

\section{Dynamic determinism} \label{sec:dyn_determinism}
This section analyzes the conditional mutual information $I(X;Y|Z)$ in the case where $X=T(Y,Z)$ is determined by $Y,Z$ via a measurable function $T:A_Y\times A_Z\to A_X$. This is a context particularly relevant to dynamics.

\subsection{Mutual information: zero or positive}
Before diving into the conditional mutual information among three random variables, we first consider two random variables. 
We begin with a trivial observation.
\begin{proposition}[Zero mutual information]\label{prop:zero_MI}
	Let $X,Z$ be two random variables. If $P_X=\delta_{x_0}$ for some $x_0\in A_X$, then 
	\[
	P_{XZ}=P_{X}\otimes P_Z.
	\]
	In particular, $I(X;Z)=\mathrm{KL}(P_{XZ}\|P_X\otimes P_Z)=0$.
\end{proposition}


%
Note that $P_X=\delta_{x_0}$ is equivalent to $X\equiv x_0$ a.s.; in this case, we may view $X=T(Z)$ for the constant map $T:z\mapsto x_0$. As we will see shortly, this is essentially the only way for $I(X;Z)$ to vanish.

More generally, consider a measurable map $T:A_Z\to A_X$ and two random variables $X,Z$. The following are equivalent
\begin{enumerate}
	\item[(i)] $X=T(Z)$ a.s.
	\item[(ii)] $\PP(X=T(Z)|Z)=1$ a.s.
	\item[(iii)] $\PP(X=T(Z)|Z=z)=1$ for $P_Z$-a.e. $z\in A_Z$.
\end{enumerate}
When one of the above holds, we say that \textit{$X$ is determined by $Z$ via $T$}.

\begin{proposition}[Positive mutual information]\label{prop:X=T(Z)gen}
	Consider a random variable $X=T(Z)$, determined by another random variable $Z$ via some measurable map $T:A_Z\to A_X$. If there is some $B_X\in\Bil_{A_X}$ with $0<P_X(B_X)<1$, then the event
	\[
	S:= B_X \times T^{-1}(A_X\setminus B_X)
	\]
	has the property that
	\[
	P_{XZ}(S)=0<P_X\otimes P_Z(S);
	\]
	in particular, we have $P_{XZ}\neq P_X\otimes P_Z$ and $I(X;Z)=\mathrm{KL}(P_{XZ}\|P_X\otimes P_Z)>0$.
\end{proposition}

\begin{proof}
	\begin{align*}
		P_{XZ}(S)=& \PP((X,Z)\in S) = \PP((T(Z),Z)\in B_X \times T^{-1}(A_X\setminus B_X)) \\
		=& \PP(Z\in T^{-1}(B_X)\cap T^{-1}(A_X\setminus B_X)) =0
	\end{align*}
	and
	\begin{align*}
		P_X\otimes P_Z(S) = & P_X(B_X) P_Z(T^{-1}(A_X\setminus B_X)) = P_X(B_X) P_X(A_X\setminus B_X) >0.
	\end{align*}
	This completes the proof.
\end{proof}


Proposition \ref{prop:X=T(Z)gen} provides a partial converse to Proposition \ref{prop:zero_MI}. If we additionally require that the alphabet $(A_X,\Bil_{A_X})$ be such that every zero-one measure is a dirac delta, then it is a complete converse.

A measure $\mu$ on a measurable space $(A,\Ail)$ is said to be a \textit{zero-one measure} if $\mu(F)$ is either 0 or 1 for all $F\in\Ail$. A dirac delta is necessarily a zero-one measure, but there are zero-one measures which are not dirac deltas. The issue usually is that the $\sigma$-algebra is too coarse.
\begin{eg}[Non-measurable singletons]
	Consider the alphabet $A_X=\{a,b\}$ equipped with the trivial $\sigma$-algebra $\Bil_{A_X}=\{\emptyset,A_X\}$. The only probability measure $P_X$ on $(A_X,\Bil_{A_X})$ is a zero-one measure, but not a dirac delta because the singletons $\{a\}, \{b\}$ are not measurable.
\end{eg}

Combining Propositions \ref{prop:zero_MI} and \ref{prop:X=T(Z)gen} yields the following dichotomy result.
\begin{theorem}[Characterization of positive mutual information]
	Let $X$ be a random variable with an alphabet $(A_X,\Bil_{A_X})$, where every zero-one measure is a dirac delta. Suppose also that $X=T(Z)$ is determined by another random variable $Z$ via a measurable map $T:A_Z\to A_X$. Then, we have a dichotomy:
	\begin{enumerate}
		\item[(i)] $X$ is constant. In this case, $I(X;Z)=0$;
		\item[(ii)] $X$ is nonconstant. In this case, $I(X;Z)>0$.
	\end{enumerate} 
\end{theorem}

Discrete spaces and Polish spaces are key examples where every zero-one measure is a dirac delta.

\begin{eg}[Separable metric space]
	If $A$ is a separable metric space, equipped with the Borel $\sigma$-algebra $\Ail$, then any zero-one measure on $(A,\Ail)$ must be a dirac delta. Indeed, if $\mu$ were a zero-one measure on $(A,\Ail)$ but not a dirac delta, then the support of $\mu$ is well-defined (see \cite[Theorem 2.1]{Parthasarathy1967}) and must contain at least two distinct points $x_1\neq x_2$ with $d(x_1,x_2)=d>0$. By definition of support, the two open balls $B_i:=B(x_i,d/3)$ are disjoint with $\mu(B(x_i,d/3))=1>0$. Now we arrive at $1=\mu(A)\geq \mu(B_1)+ \mu(B_2)= 1+1=2$, a contradiction.
	
	Specific examples include a finite or countable set $A$ equipped with the discrete distance $d(a,b)=\delta_{ab}$, and other Polish spaces equipped with the Borel $\sigma$-algebra.
\end{eg}

\subsection{Mutual information: finite or infinite}

\begin{theorem}[Mutual information: finite or infinite] \label{thm:MI_finite_infinite}
	Consider a random variable $X=T(Z)$ determined by another random variable $Z$ via some measurable map $T:A_Z\to A_X$. Assume that the singletons are measurable, i.e., $\{x\}\in\Bil_{A_X}$ for all $x\in A_X$. Then, we have a dichotomy:
	\begin{enumerate}
		\item Atomic case: there is a finite or countable set $S_X\in \Bil_{A_X}$ with $P_X(S_X)=1$. In this case, $P_{XZ} \ll P_X\otimes P_Z$ and
		\[
		I(X;Z) = \sum_{x\in S_X} P_X(x) \int_{A_Z} \ln \frac{\mathrm{d}(P_{XZ})_x}{\mathrm{d}P_Z}(z) \mathrm{d}(P_{XZ})_x(z),
		\]
		which can be either finite or infinite.
		\item Continuous case: there is $B_X\in \Bil_{A_X}$ with $P_X(B_X)>0$ and $P_X(\{x\})=0$ for all $x\in B_X$. In this case, the set
		\[
		S:= \{(T(z),z):T(z)\in B_X\}
		\]
		has the property that $P_{XZ}(S)>0$ and $P_X\otimes P_Z(S)=0$. In particular, $P_{XZ}$ is not absolutely continuous with respect to $P_X\otimes P_Z$ and hence 
		\[
		I(X;Z)=\mathrm{KL}(P_{XZ}\|P_X\otimes P_Z)=\infty.
		\]
	\end{enumerate}
\end{theorem}


\begin{proof}
	For the atomic case, consider any $N\in \Bil_{A_X}\otimes\Bil_{A_Z}$ with $P_X\otimes P_Z(N)=0$. We show $P_{XZ}(N)=0$. By Fubini, for any $x\in S_X$, we have $P_Z(N_x)=0$, where $N_x:= \{z\in A_Z: (x,z)\in N\}$. Hence,
	\begin{align*}
		P_{XZ}(N) =& \PP((X,Z)\in N) =  \sum_{x\in S_X} \PP((X,Z)\in \{x\} \times N_x) = \sum_{x\in S_X} \PP(T(Z)=x, Z\in N_x) \\
		\leq& \sum_{x\in S_X} \PP(Z\in N_x) = \sum_{x\in S_X} P_Z(N_x)=0.
	\end{align*}
	
	
	
	
	Now we show that the atomic and continuous cases form a dichotomy. If $P_X$ does not admit a $B_X$ with $P_X(B_X)>0$ and $P_X(\{x\})=0$ for all $x\in B_X$, then for every $B_X$ with $P_X(B_X)>0$, there is some $x\in B_X$ with $P_X(\{x\})>0$. Since $P_X$ is a probability measure, there can be at most countably many $x\in A_X$ with $P_X(\{x\})>0$; denote by $A^1_X$ the set of all such point atoms $x$ of $P_X$. Note $A_X^1$ is measurable because each singleton is measurable. Then by construction we must have $P_X(A_X\setminus A_X^1)=0$ because otherwise $A_X\setminus A_X^1$ would have positive measure and hence contain a point from $A_X^1$. This shows that $P_X$ has at most countable support, namely, $A_X^1$, so we are in the atomic case 1. We conclude that the two cases indeed form a dichotomy.
	
	In the atomless case 2, by definition,
	\begin{align*}
		P_{XZ}(S) = \PP((X,Z)\in S) = \PP(X=T(Z)\in B_X) = P_X(B_X)>0.
	\end{align*}
	By Fubini,
	\[
	P_X\otimes P_Z(S) = \int_{A_Z} P_X(S_z) \mathrm{d}P_Z(z) = \int_{T^{-1}(B_X)} P_X(\{T(z)\}) \mathrm{d}P_Z(z)=0,
	\]
	where in the last equality we use $T(z)\in B_X$ for any $z\in T^{-1}(B_X)$.
\end{proof}


\subsection{Infinite conditional mutual information} \label{sec:infinite_cMI}
In this section, we consider the case when $Y,Z$ together determine $X$, that is, $X=T(Y,Z)$ for a measurable map $T:A_Y\times A_Z\to A_X$. 
We split the alphabet $A_Z$ into three disjoint pieces
\[
A_Z= A_Z^0 \cup A_Z^{\mathrm{atomic}} \cup A_Z^{\mathrm{continuous}},
\]
where $A_Z^0$ consists of  $z\in A_Z$ for which the marginal distribution $P_{X_z}:=\PP(X\in \cdot|Z=z)$ of $X_z$ concentrates on a singleton, i.e., $\PP(X=x_z|Z=z)=1$ for some $x_z\in A_X$;
$A_Z^{\mathrm{atomic}}$ consists of $z\in A_Z$ for which $P_{X_z}$ concentrates on a non-singleton at most countable set, i.e., $P_{X_z}(B_X)=\PP(X\in B_X|Z=z)=1$ for some non-singleron at most countable $B_X\in\Bil_{A_X}$;
$A_Z^{\mathrm{continuous}}$ consists of $z\in A_Z$ for which $P_{X_z}$ charges an atomless continuum, i.e., there is $B_X\in \Bil_{A_X}$ with $P_{X_z}(B_X)>0$ and $P_{X_z}(\{x\})=0$ for all $x\in B_X$
By Theorem \ref{thm:MI_finite_infinite},
the three parts are disjoint and indeed form a partition of $A_Z$.

\begin{theorem}[Conditional mutual information] \label{thm:cMI_0_infty}
	Let random variable $X=T(Y,Z)$ be determined by random variables $Y,Z$ via a measurable map $T:A_Y\times A_Z\to A_X$. Suppose $X,Y,Z$ all have standard alphabets. Then,
	\begin{align*}
		I(X;Y|Z) 
		= \begin{cases}
			\int_{A_Z^{\mathrm{atomic}}} I(X_z;Y_z)\mathrm{d}P_Z(z)& \text{if }P_Z(A_Z^{\mathrm{continuous}})=0,\\
			\infty&\text{else.}
		\end{cases}
	\end{align*}
	In particular, when $P_Z(A_Z^0)=1$, we have $I(X;Y|Z)=0$.
\end{theorem}

\begin{proof}
	By Proposition \ref{prop:average_disintegrated_MI}, we split the conditional mutual information into three parts.
	\[
	I(X;Y|Z) = \int_{A_Z^0} I(X_z;Y_z)\mathrm{d}P_Z(z) + \int_{A_Z^{\mathrm{atomic}}} I(X_z;Y_z)\mathrm{d}P_Z(z) + \int_{A_Z^{\mathrm{continuous}}} I(X_z;Y_z)\mathrm{d}P_Z(z)
	\]
	By Proposition \ref{prop:zero_MI}, $I(X_z;Y_z)=0$ for any $z\in A_Z^0$, so the first term vanishes. The last term is zero when $P_Z(A_Z^{\mathrm{continuous}})=0$ and is $\infty$ otherwise, according to Theorem \ref{thm:MI_finite_infinite}. 
\end{proof}

In many dynamically relevant situations, we have $P_Z(A_Z^{\mathrm{continuous}})>0$, as announced in Theorem A, and hence $I(X;Y|Z)=\infty$.

\section{Examples}\label{sec:examples}
\subsection{Bernoulli interval maps}\label{sec:Bernoulli_maps}
Consider the piecewise linear expanding map $E_d:[0,1]\to[0,1]$, $d\in\Z$, $d\geq 2$, on the unit interval given by $E_d(x)=d\cdot x\mod 1$. 

If $Y$ is uniformly distributed on the interval, i.e., $P_Y=\mathrm{Leb}_{[0,1]}$, then so is $X=E_d(Y)$, i.e., $P_X=P_Y$. The joint distribution of $(X,Y)$ on the unit square $[0,1]^2$ is given by $P_{XY}=(E_d,\mathrm{id})_*\mathrm{Leb}_{[0,1]}$, which is supported on the graph of $E_d$. In particular, $P_{XY}$ is mutual singular with respect to $P_X\otimes P_Y=\mathrm{Leb}_{[0,1]^2}$. By Theorem \ref{thm:MI_finite_infinite}, $I(X;Y)=\mathrm{KL}(P_{XY}\| P_X\otimes P_Y)=\infty$. 

Now we discretize. Fix a positive integer $L=\Delta^{-1}\in \Z_{>0}$. Then, the uniform partition by $\{\left[\frac{i-1}{L}, \frac{i}{L} \right)\}$
is a Markov partition for $E_d$. 
Note
\[
\PP(X^{\Delta}=i\Delta)=  \mathrm{Leb}_{[0,1]}\left[\frac{i-1}{L}, \frac{i}{L} \right) = \frac{1}{L},~~~~\forall i=1,\cdots,L.
\]
This shows that $X^{\Delta}$ is uniformly distributed on $\{1,\cdots,L\}$. So is $Y^{\Delta}= \Pi^{\Delta}Y$.
The joint distribution $P_{X^{\Delta}Y^{\Delta}}$ of $(X^{\Delta},Y^{\Delta})$ charges uniform mass 
to the pairs
\begin{equation} \label{eq:Bernoulli_map_ij}
	(d(i-1)+r\mod L,i),~~~~i=1,\cdots,L,~~r=1,\cdots,d.
\end{equation}
When $L\leq d$, then $P_{X^{\Delta}Y^{\Delta}}$ is uniform on $\{1,\cdots,L\}^2$, with $P_{X^{\Delta}Y^{\Delta}}(i,j)=\frac{1}{L^2}$ for each $(j,i)\in\{1,\cdots,L\}^2$. In this case,
\[
I(X^{\Delta};Y^{\Delta}) = \mathrm{KL}(P_{X^{\Delta}Y^{\Delta}} \| P_{X^{\Delta}}\otimes P_{Y^{\Delta}})=0. 
\]

When $L>d$, then only $dL$ pairs of $(j,i)\in \{1,\cdots,L\}^2$ satisfying eq. (\ref{eq:Bernoulli_map_ij}) are charged with mass $\frac{1}{dL}$ each. In this case, 
\begin{align*}
	I(X^{\Delta};Y^{\Delta}) =& \mathrm{KL}(P_{X^{\Delta}Y^{\Delta}}\| P_{X^{\Delta}}\otimes P_{Y^{\Delta}}) \\
    =& \sum_{(j,i)} P_{X^{\Delta}Y^{\Delta}}(j,i) \ln \frac{P_{X^{\Delta}Y^{\Delta}}(j,i)}{P_{X^{\Delta}}\otimes P_{Y^{\Delta}}(j,i)} \\
    =& dL \frac{1}{dL} \ln \frac{1/dL}{1/L^2} = \ln L-\ln d.
\end{align*}

In the discretized version, a more expanding map $E_d$ with large $d$ gives less mutual information.

\subsection{Sine box functions} \label{sec:sine_box}
Consider the sine box function $S_n:[0,1]\to[0,1]$ given by
\[
S_n(x):=\frac{1+\sin2\pi nx}{2},~~~~n=1,2,\cdots.
\]
We compute its invariant measure $\mu_n$ by taking a long trajectory $\{x_t=S_n^t(x_0):t=\tau_0,\tau_0+1,\cdots,\tau_0+\tau-1\}$, starting from $x_0=0.5$ (other initial points $0.2,0.3,\cdots,0.9$ yielded very similar results), discarding the first $\tau_0=1000$ iterates as transient, and collecting the next $\tau=10^6$ iterates to approximate
\[
\mu_n \approx \mu_n^{(\tau)}:=\frac{1}{\tau}\sum_{t=\tau_0}^{\tau_0+\tau-1} \delta_{x_t}.
\]
If $Y$ follows $\mu_n$, then $X=S_n(Y)$ follows $(S_n)_*\mu_n=\mu_n$.

The probability density function $\phi$ of $\mu_n$ is approximated by the histogram for $\{x_t\}$ binned into $\{\left[(i-1)\Delta, i\Delta\right)\}$, that is,
\[
\phi^{(\tau)}((i-1)\Delta) := \frac{1}{\tau} \sum_{t=\tau_0}^{\tau_0+\tau-1} \One_{\left[(i-1)\Delta, i\Delta\right)}(x_t),~~~~i=1,\cdots,L,
\]
which can be represented in vector form
\[
\phi^{(\tau)} = (\phi^{(\tau)} _i)_{i=1}^L,~~~~\phi^{(\tau)}_i := \phi^{(\tau)}((i-1)\Delta).
\]
The product of the marginals $P_X\otimes P_Y$ discretizes into $P_{X^{\Delta}}\otimes P_{Y^{\Delta}} = (\Pi^{\Delta}\times \Pi^{\Delta})(P_{X}\otimes P_Y)$, which is approximated by
\[
P_{X^{\Delta}}\otimes P_{Y^{\Delta}}  \approx P_{X^{\Delta}}^{(\tau)}\otimes P_{Y^{\Delta}}^{(\tau)}:=\phi^{(\tau)}\cdot (\phi^{(\tau)})^{\top}
\]

The joint distribution $P_{XY}$ discretizes into $P_{X^{\Delta}Y^{\Delta}}=(\Pi^{\Delta}, \Pi^{\Delta})(P_{XY})$, which is then approximated by
\[
P_{X^{\Delta}Y^{\Delta}} \approx  P_{X^{\Delta}Y^{\Delta}}^{(\tau)} = ( P_{X^{\Delta}Y^{\Delta}}^{(\tau)})_{i,j=1}^L,
\]
where 
\[
( P_{X^{\Delta}Y^{\Delta}}^{(\tau)})_{i,j} = \frac{1}{\tau} \sum_{t=\tau_0}^{\tau_0+\tau-1} \One_{[(i-1)\Delta,i\Delta)}(x_t)\cdot \One_{[(j-1)\Delta,j\Delta)}(x_{t+1}).
\]
It follows from Theorem \ref{thm:MI_finite_infinite} that $I(X,Y)=\infty$ for any $n$.  However, higher value of $n$ decreases the ability to resolve uncertainty about $X=S_n(Y)$ from knowledge about $Y$. Accordingly, we expect $I(X^{\Delta};Y^{\Delta})$ to decrease as $n$ increases. This is confirmed by simulations as shown in Figure \ref{fig:sine_box}.

\newpage
\appendix

\section{Regular conditional probability and disintegration on standard measurable spaces} \label{sec:RCP_standard_space}
We motivate the consideration of standard measurable spaces by an attempt to generalize the definition of conditional probability for discrete variables to more general variables. We finish the discussion by showing that regular conditional probabilities are equivalent to disintegrations in our setting.

Consider a common probability space $(\Omega,\Fil,\PP)$ for random variables $X,Y,Z$ taking values in $(A_X,\Bil_{A_X})$, $(A_Y,\Bil_{A_Y})$, $(A_Z,\Bil_{A_Z})$.

The first challenge in generalizing the definition of conditional probability $\PP(F|Z=z):=\frac{\PP(F\cap\{Z=z\})}{\PP(Z=z)}$ to non-discrete variables $Z$ is that the events $\{Z=z\}$ being conditioned on may well have zero probability.
To overcome this challenge, a first fix is to interpret the conditional probability as a density (Radon-Nikodym derivative) rather than a fraction. More precisely, given an arbitrary random variable $Z$ and a fixed event $F\in\Fil$,
we define $\PP(F|Z=z)$, $z\in A_Z$ to be the Radon-Nikodym derivative
\[
\PP(F|Z=z):= \frac{\mathrm{d}\PP^F(Z\in\cdot)}{\mathrm{d}\PP(Z\in\cdot)}(z) = \frac{\mathrm{d}\PP^F(Z\in\cdot)}{\mathrm{d}P_Z}(z) ,~~~~z\in A_Z,
\]
where $\PP^F(Z\in \cdot):=  \PP(F\cap \{Z\in \cdot\})$ is absolutely continuous with respect to $P_Z=\PP(Z\in\cdot)$. By Radon-Nikodym Theorem, $\PP(F|Z=z)$ exists and is $P_Z$-essentially unique.
Equivalently, we have the defining equation for $\PP(F|Z=z)$
\[
\PP(F\cap\{Z\in B_Z\}) = \PP^F(Z\in B_Z)=\int_{B_Z} \PP(F|Z=z)\mathrm{d}P_Z(z),~~~~B_Z\in \Bil_{A_Z},
\]
an analogue of the discrete alphabet case
\[
\PP(F\cap\{Z\in B_Z\})=\sum_{z\in B_Z} \PP(F|Z=z) P_Z(z),~~~~B_Z\in \Bil_{A_Z}.
\]

This is a more direct construction than the usual conditioning on sigma-algebra, which we review below for comparison.
For a fixed event $F\in\Fil$, the \textit{conditional probability} $\PP(F|\Gil)$ given a sigma-algebra $\Gil\subseteq\Fil$ is defined to be any $\Gil$-measurable random variable $g:\Omega\to[0,1]$ with 
\[
\int_G g\mathrm{d}\PP = \PP(F\cap G),~~~~\forall G\in\Gil.
\]
$\PP(F|\Gil)$  exists and is $\PP$-a.s. unique as the Radon-Nikodym derivative of $\PP(F\cap \cdot)/\PP(F)$ with respect to $\PP$, both restricted to $\Gil$, provided $\PP(F)>0$; in case $\PP(F)=0$, we have $\PP(F|\Gil)\equiv0$.
Now consider $\Gil=\sigma(Z)$.
Since $\PP(F|\sigma(Z))$ is $\sigma(Z)$-measurable, it can be factored through $Z$ \cite[Lemma 5.2.1]{Gray2009}, that is, 
\[
\PP(F|\sigma(Z))= h\circ Z,
\]
for some measurable function $h:A_Z\to[0,1]$. We thus have
\[
\PP(F|Z=z)= h(z).
\]


A subtle issue remains with this Radon-Nikodym construction, namely, the potential pile up of exceptional sets $E(F)$ in the definition of $\PP(F|Z=z)$. The Radon-Nikodym derivative $\PP(F|Z=z)$ is well-defined up to an exceptional set $E(F)$ with $\PP(E(F))=0$ depending on the event $F$. These exceptional sets may pile up $\PP\left(\bigcup_{F\in\Fil} E(F)\right)=1$ and in this case we cannot define $\PP(F|Z=z)$ simultaneously for all $F\in\Fil$.
An example of such a pathology can be found in \cite[Page 624]{Doob1990}; for more details see \cite[Chapter 5.1.3]{Durrett2010}.
Hence, in order to generalize the definition of $P_{X\times Y|Z}$ as in Eq. (\ref{eq:cMI_discrete_defn}), we need to rule out such pathologies. This motivates our second fix: the regular conditional probability.
\begin{defn}[Regular conditional probability (RCP); \cite{Gray2009} Chapter 5.8]
The \textit{regular conditional probability} given a sub-$\sigma$-algebra $\Gil\subseteq\Fil$ is a function $f:\Fil\times \Omega\to[0,1]$ such that
	\begin{enumerate}
		\item for each $\omega\in\Omega$, $f(\cdot,\omega)$ is a probability measure on $(\Omega,\Fil)$;
		\item for each $F\in\Fil$, $f(F,\cdot)$ is a version of $\PP(F|\Gil)$.
	\end{enumerate}
We consider sigma-algebra $\Gil=\sigma(Z)$ and events of the form $F=\{X\in B_X\}\in\Fil$.
Define the \textit{regular conditional distribution} of $X$ given $Z$ to be
	\[
	\PP(X\in B_X|Z=z) := f(\{X\in B_X\},\omega),~~\omega\in Y^{-1}\{z\}.
	\]
\end{defn}

    RCP does not always exist in general but it does, for example, \cite[Corollary 5.8.1]{Gray2009}  (i) when both $(A_X,\Bil_{A_X})$ and $(A_Z,\Bil_{A_Z})$ are standard, (ii) when either is discrete.

\begin{defn}[Standard measurable space; \cite{Arnold1998} page 541]
	A measurable space $(\Omega,\Fil)$ is called a \textit{standard measurable space} if isomorphic via a bi-measurable bijection to a Borel subset of a Polish space.
\end{defn}

In particular, a standard measurable space $(\Omega,\Fil)$ admits
a sequence of finite fields $\Fil_n\subseteq \Fil$, $n=0,1,\cdots$ such that
\begin{enumerate}
	\item increasing fields: $\Fil_{n}\subseteq \Fil_{n+1}$ for all $n=0,1,\cdots$;
	\item generating fields: $\Fil=\sigma\left(\bigcup_{n=0}^{\infty} \Fil_n\right)$;
	\item nonempty atomic intersection: an event is called an \textit{atom} of a field if it is nonempty and its only subsets which are members of the field are the empty set and itself. If $G_n\in\Fil_n$, $n=0,1,\cdots$ are atoms with $G_{n+1}\subseteq G_n$ for all $n$, then
	\[
	\bigcap_{n=0}^{\infty}G_n\neq\emptyset.
	\]
\end{enumerate}
In fact, the above three conditions are sometimes taken to be the defining properties of a standard measurable space, for example in \cite{Gray2011}. We have taken the more restricted definition of Arnold \cite{Arnold1998} to ensure that both regular conditional probabilities and disintegrations exist.


Now we review disintegrations and show that they coincide with regular conditional probabilities in our setting.
\begin{defn}[Disintegration; \cite{Arnold1998} pp 22]
Given a probability measure $\mu$ on a product measurable space $(A\times B,\Ail\otimes \Bil)$ and a probability measure $\nu$ on $(A,\Ail)$, we say that a function $\mu_{\cdot}(\cdot):A\times \Bil\to[0,1]$ is a \textit{disintegration} of $\mu$ with respect to $\nu$ if 
\begin{enumerate}
	\item for all $B\in\Bil$, $a\mapsto\mu_a(B)$ is measurable function from $(A,\Ail)$ to $([0,1],\Bil([0,1]))$;
	\item for $\nu$-a.e. $a\in A$, $B\mapsto \mu_a(B)$ is a probability measure on $(B,\Bil)$;
	\item for all $E\in \Ail\otimes \Bil$, 
	\[
	\mu(E) = \int_A\int_B \One_E(a,b)\mathrm{d}\mu_a(b) \mathrm{d}\nu(b).
	\]
\end{enumerate}
\end{defn}

Disintegrations do not always exist, but they do exist $\nu$-essentially uniquely, when $(A,\Ail)$, $(B,\Bil)$ are both standard alphabets, see \cite[Proposition 1.4.3]{Arnold1998} and \cite[Corollary 5.8.1]{Gray2009}. 

Returning to our previous setting, $P_{XYZ},P_{X\times Y|Z}$ both have $Z$-marginals equal to $P_Z$ by construction and so both admit disintegrations with respect to $P_Z$ denoted by $(P_{XYZ})_z$ and $(P_{X\times Y|Z})_z$. In this case, it follows from the definitions of RCP and disintegration and their existence and essential uniqueness that for $P_Z$-a.e. $z\in A_Z$, and all $B_X\in\Bil_{A_X},B_Y\in\Bil_{A_Y}$, we have
\begin{align*}
	(P_{XYZ})_z(B_X\times B_Y) =& \PP(X\in B_X,Y\in B_Y|Z=z),\\
	(P_{X\times Y|Z})_z(B_X\times B_Y) =& \PP(X\in B_X|Z=z)\PP(Y\in B_Y|Z=z).
\end{align*}

\section{Additive noise}\label{sec:additive_noise}
Consider a measurable map $T_0:[0,1]\to[0,1]$ on the unit interval, which is \textit{nonsingular} with respect to the Lebesgue measure $\lambda$ on $[0,1]$ in the sense that $\lambda(T_0^{-1}N)=0$ for any $\lambda(N)=0$.
Consider random variable $Z$ with distribution $P_Z= h_Z \lambda$.

Let $X_0=T_0(Z)$. 
By Theorem \ref{thm:MI_finite_infinite}, we have $I(X_0;Z)=\infty$.

Now perturb $T_0$ by additive noise
\[
T_{\xi}:z\mapsto T_0(z)+\xi\mod1,
\]
where the noise $\xi$ is independent of $Z$ and follows some distribution $P_{\xi}= h_{\xi} \lambda$.

For concreteness, we take the uniform noise of amplitude $\epsilon$ centered at 0 with density $h_{\xi}=\frac{1}{\epsilon}\One_{[-\epsilon/2,\epsilon/2]}$.

Consider $X$ given by the randomly transformed $Z$ via $\{T_{\xi}\}$; more precisely,
\[
\PP(X\in B|Z=z) =\int_0^1 \One_B\circ T_{\xi}(z) \mathrm{d}P_{\xi}(\xi).
\]
In other words,
\[
(P_{XZ})_z = (R_{T_0(z)})_*P_{\xi},~~~~R_{\alpha}: x\mapsto x+\alpha\mod1.
\]

If the joint distribution $P_{XZ}\ll P_X\otimes P_Z$, then
\begin{align*}
	I(X;Z) =& \int_{[0,1]^2} f\ln f\mathrm{d} P_X\otimes P_Z,
\end{align*}
where $f(x,z)=\frac{\mathrm{d} P_{XZ}}{\mathrm{d}P_X\otimes P_Z}(x,z) = \frac{\mathrm{d}(P_{XZ})_z}{\mathrm{d}(P_X\otimes P_Z)_z}(x) = \frac{\mathrm{d}(\frac{1}{\epsilon} \One_{[T_0(z)-\epsilon/2,T_0(z)+\epsilon/2]} )\lambda}{\mathrm{d}P_X}(x)$. 

In general, $f$ depends on $T_0$. Consider the special case of Bernoulli maps $T_0=E_d$ or roations $T_0=R_{\alpha}$, both of which preserve $\lambda$. Then, $P_X=\lambda$, $f(x,z)=\frac{1}{\epsilon} \One_{[T_0(z)-\epsilon/2,T_0(z)+\epsilon/2]}(x)$, and we have
\begin{align*}
	I(X;Z) =& \int_{[0,1]^2} \frac{1}{\epsilon} \One_{[T_0(z)-\epsilon/2, T_0(z)+\epsilon/2]}(x) \ln \frac{1}{\epsilon} \One_{[T_0(z)-\epsilon/2, T_0(z)+\epsilon/2]}(x) \mathrm{d}x\mathrm{d}z\\
	=&\int_0^1 \int_{T_0(z)-\epsilon/2}^{T_0(z)+\epsilon/2} \frac{1}{\epsilon} \ln \frac{1}{\epsilon} \mathrm{d}x \mathrm{d}z \\
	=& \epsilon\frac{1}{\epsilon} \ln\frac{1}{\epsilon} =\ln \frac{1}{\epsilon}.
\end{align*}
This indicates that the mutual information of the blurred variables does not distinguish between very ambiguous map $T_0=E_d$ and non-ambiguous map $T_0=R_{\alpha}$.



\section{Derivation of the discretized mutual information formula} \label{sec:app_discretized_MI_formula}

Recall that the Shannon entropy of a continuous random variable $X$ is infinite, but there is a meaningful notion of differential entropy, which differs from the Shannon entropy of the discretization of $X$ by an infinite offset.

In a similar spirit, we aim to identify such an infinite offset in mutual information $I(X;Y)$ with $X=T(Y)$ so as to extract the meaningful term $A_T(Y)$, which we have termed the relative ambiguity of the system $(T,Y)$.

Observe that $P_{X^{\Delta}Y^{\Delta}} \ll P_{X^{\Delta}}\otimes P_{Y^{\Delta}}$ and hence
\begin{equation}\label{eq:I(X^Delta;Y^Delta)}
	I(X^{\Delta};Y^{\Delta}) = \sum_{i,j} \PP(X^{\Delta}=i\Delta,Y^{\Delta}=j\Delta)  \ln \frac{\PP(X^{\Delta}=i\Delta,Y^{\Delta}=j\Delta)}{\PP(X^{\Delta}=i\Delta) \PP(Y^{\Delta}=j\Delta)}.
\end{equation}
Since the densities $f_X,f_Y$ are continuous by assumption in Conjecture C, we have the usual Riemman sum approximation
\begin{align*}
	\PP(X^{\Delta}=i\Delta) \approx & f_X(i\Delta)\Delta\\
	\PP(Y^{\Delta}=j\Delta) \approx & f_Y(j\Delta)\Delta.
\end{align*}
In the linear case $T=E_d$, the mass $f_X(i\Delta)\Delta$ splits evenly into $d=|T'(i\Delta)|$ pieces. Since $T$ is piecewise $C^1$ expanding $|T'|\geq 1$ by assumption in Conjecture C, we conjecture the key approximation
\[
	\PP(X^{\Delta}=i\Delta, Y ^{\Delta}=j\Delta) \approx \frac{f_X(i\Delta)\Delta}{|T'(i\Delta)|},~~~~T(i\Delta)\approx j\Delta.
\]
When $T$ has contracting regions $|T'|<1$, this approximation fails.
This suggests a connection to the transfer operator formula for expanding maps
\[
(\hat{T}f)(y)= \sum_{x\in T^{-1}y} \frac{f(x)}{|T'(x)|},
\]
where the transfer operator $\hat{T}:L^1(\lambda)\to L^1(\lambda)$ is defined to be 
the Radon-Nikodym derivative
\[
\hat{T} f := \frac{\mathrm{d} T_*(f\lambda)}{\mathrm{d}\lambda}.
\]
Now we combine these approximations together:
\begin{align*}
	I(X^{\Delta},Y^{\Delta}) =&  \sum_{i\Delta, j\Delta} \PP(X^{\Delta}=i\Delta,Y^{\Delta}=j\Delta)  \ln \frac{\PP(X^{\Delta}=i\Delta,Y^{\Delta}=j\Delta)}{\PP(X^{\Delta}=i\Delta) \PP(Y^{\Delta}=j\Delta)}\\
	\approx& \sum_{j\Delta} \sum_{i\Delta\in T^{-1}j\Delta} \frac{f_X(i\Delta) \Delta}{|T'(i\Delta)|} \ln \frac{f_X(i\Delta)\Delta / |T'(i\Delta)|}{f_X(i\Delta)\Delta f_Y(j\Delta)\Delta}\\
	=& \sum_{j\Delta} \sum_{i\Delta\in T^{-1}j\Delta} \frac{f_X(i\Delta) \Delta}{|T'(i\Delta)|} \ln \frac{1}{ |T'(i\Delta)| f_Y(j\Delta)\Delta}\\
	\approx& \int_{A_Y} \hat{T}\left[f_X \ln \frac{1}{|T'|\cdot f_Y\circ T \cdot \Delta}\right] \mathrm{d}y\\
	=& \int_{A_X} f_X \ln \frac{1}{|T'|}\mathrm{d}x + \int_{A_X} f_X \ln \frac{1}{f_Y\circ T} \mathrm{d}x + \int_{A_X} f_X \ln\frac{1}{\Delta} \mathrm{d}x\\
	=& -\int_{A_X} f_X \ln |T'| \mathrm{d}x + \int_{A_Y} (\hat{T}f_X) \ln \frac{1}{f_Y} \mathrm{d}y + \ln \Delta^{-1} \\
	=& H(Y) - \int \ln |T'|\mathrm{d}P_X + \ln \Delta^{-1}.
\end{align*}

\newpage
\bibliographystyle{amsalpha}
\bibliography{references}

\end{document}